\documentclass[12pt,russian]{article}
\usepackage{mathrsfs,amsfonts,amsmath,amssymb,amsthm}
\usepackage{latexsym}
\usepackage{hyperref}
\usepackage[T2A]{fontenc}
\usepackage[english]{babel}
\usepackage{xcolor}
\usepackage{tikz}
\usetikzlibrary{positioning}

\newtheorem{theorem}{Theorem}
\newtheorem{lemma}{Lemma}
\newtheorem{definition}{Definition}

\newtheorem{corollary}{Corollary}
\newtheorem{remark}{Remark}

\newtheorem{observation}{Observation}
\newtheorem{notation}{Notation}
\newtheorem{conjecture}{Conjecture}

\renewcommand{\geq}{\geqslant}
\renewcommand{\leq}{\leqslant}

\def\_phi{\varphi}

\title{On  forest and bipartite cuts \\ in sparse graphs}
\author{Ilya I. Bogdanov$^1$, Elizaveta Neustroeva$^2$,  Georgy Sokolov$^3$, \\ Alexey Volostnov$^4$,  Nikolay Russkin$^2$, Vsevolod Voronov$^5$}

\begin{document}


\maketitle

\footnotetext[1]{PhD, Moscow Institute of Physics and Techology, associate professor}
\footnotetext[2]{Moscow Institute of Physics and Techology, graduate student}
\footnotetext[3]{Moscow Institute of Physics and Techology, postgraduate student}
\footnotetext[4]{Moscow Institute of Physics and Techology, assistant professor}
\footnotetext[5]{PhD, Moscow Institute of Physics and Techology, Caucasus Mathematical Center of Adyghe State University, senior researcher}
    
\begin{abstract}
The paper is devoted to sufficient conditions for the existence of vertex cuts in simple graphs, where the induced subgraph on the cut vertices belongs to a specified graph class. 
In particular, we show that any connected graph with $n$ vertices and fewer than $(19n - 28)/8$ edges admits a forest cut. This result improves upon recent bounds, although it does not resolve the conjecture that the sharp threshold is $3n - 6$ (Chernyshev, Rauch, Rautenbach, 2024).  Furthermore, we prove that if the number of edges is less than $(80n-134)/31$, then the graph admits a bipartite cut.
\end{abstract}

\section{Introduction}

Let $G = (V, E)$ be a connected undirected graph without loops or multiple edges. A subset $M \subset V$ is called a \emph{vertex cut} if the removal of $M$ disconnects the graph $G$.

We denote by $G|_M$ the subgraph of $G$ induced on the vertex set $M$. This notation is used instead of the standard $G[M]$ to improve the readability of formulas. If $G|_M$ is empty (i.e., $M$ is an independent set), we say that $G$ has an \emph{independent cut}. Similarly, if $G|_M$ is a forest or a bipartite graph, we refer to $G$ as having a \emph{forest cut} or a \emph{bipartite cut}, respectively.

The problem considered in this paper originates from the theory of separators in graphs, a well-established area with numerous significant results for various classes of graphs~\cite{lipton1977applications, gilbert1984separator, alon1990separator, fox2010separator}. Separator theorems typically state that removing a subset $M \subset V$ leaves any connected component of $G|_{V \setminus M}$ with less than $cn$ vertices, where $0 < c < 1$,  and that the cardinality of the separator $M$ is $O(\sqrt{n})$ or $O(\sqrt{n} \log n)$. These results play an important role in computational complexity, as they allow large problems to be recursively reduced to smaller instances~\cite{rosenberg2005graph}.

In the paper by Y. Caro and R. Yuster \cite{caro1999graph}, it was shown that the algorithmic complexity of decomposing a graph $X$ into edge-disjoint subgraphs isomorphic to a given graph $G$ depends on whether $G$ has an independent cut. In \cite{caro1999graph}, the term \emph{robust} was used for graphs without an independent cut. Furthermore, Y. Caro formulated a conjecture, which was later proved in \cite{chen2002note}. We give it as a theorem.

\begin{theorem}[G. Chen, X. Yu, 2002]
\label{thm:independent}
Every graph on $n$ vertices with at most $2n-4$ edges has an independent cut.
\end{theorem}

In \cite{chen2002note,chen2002fragile}, the term ``fragile'' was used for graphs with an independent cut. The papers \cite{pfender2013,rauch2024} showed that all minimal robust graphs, i.e., graphs with $2n-3$ edges that become fragile after the removing  any edge, can be obtained by gluing triangles $K_3$ and triangular prism graphs along edges and triangles.

The following conjecture was formulated in \cite{chernyshev2024forestcutssparsegraphs}:

\begin{conjecture}[Chernyshev, Rauch, Rautenbach, 2024]
    \label{cnj:main}
    If $G$ is a graph of order $n$ with fewer than $3n-6$ edges, then $G$ has a forest cut.
\end{conjecture}

The work \cite{chernyshev2024forestcutssparsegraphs} also noted that the conjecture holds for planar graphs, and planar triangulations with $3n-6$ edges are one example of critical graphs in this problem. Furthermore, a family of non-planar critical graphs can be obtained by gluing triangulations along $K_3$ and $K_4$ cliques. These graphs have no forest cut, but a forest cut appears after the removal of any edge.

Over the past year, various authors have obtained linear bounds weaker than those required by the conjecture. Below are the edge count constraints under which the existence of a forest cut has been proven:

\begin{itemize}
    \item $|E|<\frac{11n-18}{5}$ \cite{chernyshev2024forestcutssparsegraphs},    
    \item $|E|<\frac{9n-15}{4}$ \cite{botler2024extremalproblemsforestcuts},
    \item $|E|\leq 2n$ \cite{cheng2024sparsegraphsindependentforesty}.
\end{itemize}

In this paper, we prove the following bound:

\begin{theorem}
\label{thm:forest}
    If $G$ is a graph of order $n\geq 4$ with fewer than $\frac{19n-28}{8}$ edges, then $G$ has a forest cut.
\end{theorem}

In addition to forest cuts, we also consider the case of bipartite cuts. Note that if Conjecture~\ref{cnj:main} holds, then the edge bound that implies the existence of a bipartite cut cannot differ from that for a forest cut. Indeed, a 3-tree on $n$ vertices has exactly $3n - 6$ edges, and any vertex cut in a 3-tree necessarily contains a triangle, hence cannot induce a bipartite graph. Motivated by this observation, we propose the following relaxed 

\begin{conjecture}
If a graph $G$ with $n \geq 4$ vertices has fewer than $3n - 6$ edges, then $G$ admits a bipartite cut.
\end{conjecture}

The bound we are able to establish in the case of bipartite cuts is weaker than the one proposed in Conjecture~\ref{cnj:main}, but it is significantly better than the bound for forest cuts:

\begin{theorem}
\label{thm:bipartite}
Let $G$ be a graph on $n \geq 4$ vertices. If $G$ has fewer than $\frac{80n - 134}{31}$ edges, then $G$ has a bipartite cut.
\end{theorem}

As in~\cite{chernyshev2024forestcutssparsegraphs}, our bounds are derived through the analysis of a linear programming problem. However, our approach differs in that, in Sections~\ref{sec:prop_cut} and~\ref{sec:min_counter}, we establish structural properties of minimal counterexamples across an infinite family of related problems, including the assumptions of Conjecture~\ref{cnj:main}. Sections~\ref{sec:forest} and~\ref{sec:bipartite} then present additional constructions and constraints specific to the forest and bipartite cases, respectively. In proving the main results, we demonstrate the infeasibility of the associated system of inequalities for the selected parameter values.

To the best of our knowledge, the problem of determining a linear edge bound that guarantees the existence of a bipartite cut has not been previously considered.

Note that a related problem was studied in~\cite{bonsma2012extremal}, where it was shown that any graph with $n$ vertices and fewer than $\lceil (3n - 1)/2 \rceil$ edges admits an \emph{edge cut} that forms a matching. Furthermore, the paper~\cite{bonsma2012extremal} provides a classification of the extremal graphs with exactly $\lceil (3n - 1)/2 \rceil$ edges for which no such matching cut exists.

	\section{Preliminaries}

We use standard notation, except that $G|_M$ denotes the induced subgraph on the vertex set $M$.

Let $V(G)$ denote the vertex set of the graph $G$, $E(G)$ its edge set, $K_n$ the complete graph on $n$ vertices, and $C_n$ the cycle on $n$ vertices. For convenience, we introduce a notation for some small graphs. Let $T_0$ be the 4-vertex graph consisting of a triangle and an isolated vertex, $T_1$ the 4-vertex graph consisting of a triangle with an extra vertex connected to a triangle vertex, and $D$ the ``diamond graph'' $K_4 \setminus e$ (Fig.~\ref{fig:4vtx}). Let $N(v)$ denote the neighborhood of vertex $v$ (the relevant graph is usually clear from the context). Let $|G|$ be the number of vertices in the graph $G$, and $e(G)$ be the number of edges \cite{bollobas2013modern}.

\begin{figure}[ht]
    \centering
    \includegraphics[width=8cm]{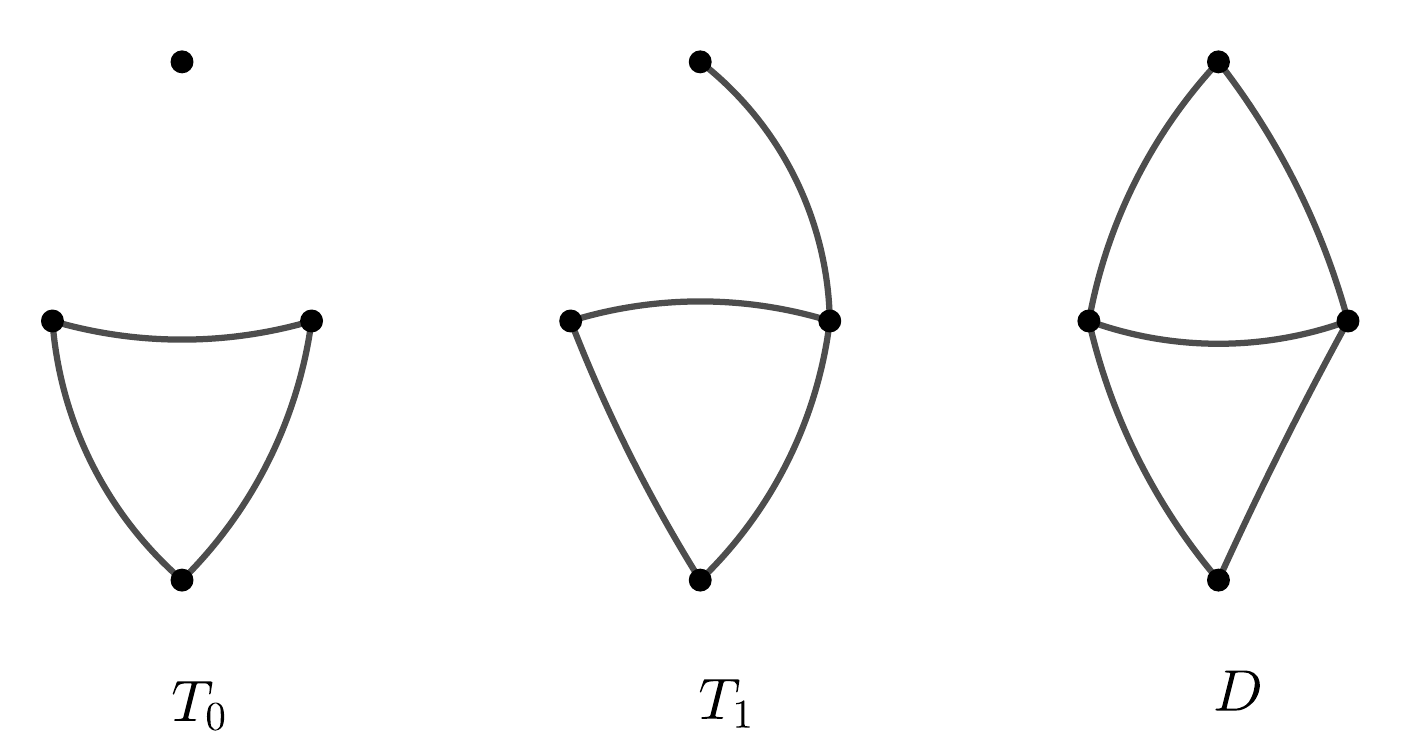}
    \caption{Three graphs that will appear as induced subgraphs on the cut}
    \label{fig:4vtx}
\end{figure}

Since our arguments will frequently involve considering (not necessarily connected) subgraphs resulting from $G \setminus M$, we introduce the following:

\begin{definition}
    A separation of a graph $G$ is a triple  $(M, L, R)$ satisfying: \[V(G) = M \sqcup L \sqcup R,\] where the vertex sets $L$ and $R$ are non-empty, and no vertex in $L$ is adjacent to any vertex in $R$. Thus, $M$ is a vertex cut of $G$, while each of $L$ and $R$ represents the union of vertices from one or more connected components into which $M$ splits $G$.
\end{definition}

We will use the following terminology: $M$ is called a \emph{cut}, while the triple $(M,L,R)$ is called a \emph{separation}. Unlike studies on graph separators~\cite{caro1999graph,rosenberg2005graph}, we do not impose bounds on \( |M|, |L|, |R| \). For fragile graphs, such estimates  were studied in \cite{chen2002fragile}.

\begin{definition}
    A quality function parameterized by  $\alpha, \beta \in \mathbb{R}$, is defined as:
    \[
    q_{\alpha,\beta}(G) = \alpha |G| - e(G) - \beta.
    \]
\end{definition}

A crucial role in our arguments is played by the set of parameter pairs $(\alpha, \beta)$ satisfying:
\begin{equation}   
\label{eq:cond_ab}
    2 < \alpha \leq 3, \quad
    4\alpha - \beta = 6.    
\end{equation}

This set forms a half-interval on the line $4\alpha - \beta = 6$, including the point $(3,6)$ but excluding $(2,2)$. Some statements are proved for any pair $(\alpha,\beta)$ satisfying \eqref{eq:cond_ab}. Occasionally, we will specialize to cases relevant to our main results, namely:

\begin{itemize}
    \item $(\alpha, \beta) = \left(\frac{19}{8}, \frac{28}{8}\right)$ for forest cuts (Theorem \ref{thm:forest}),
    \item $(\alpha, \beta) = \left(\frac{80}{31}, \frac{134}{31}\right)$ for bipartite cuts (Theorem \ref{thm:bipartite}),
    \item $(\alpha, \beta) = (3, 6)$ in Conjecture \ref{cnj:main}.
\end{itemize}

In these three cases, conditions $\eqref{eq:cond_ab}$ are satisfied. Therefore, statements proved under the assumption of \eqref{eq:cond_ab} will hold in our cases of interest.

Let $\mathfrak{F}$ denote the set of all forests on any number of vertices, and $\mathfrak{B}$ the set of bipartite graphs. We will need the following:

\begin{observation}
    \label{obs:hereditary}
    Forests $\mathfrak{F}$ and bipartite graphs $\mathfrak{B}$ are hereditary graph classes, i.e. every subgraph of a graph in the class also belongs to the class. Furthermore, these classes are closed under gluing along the edge.
\end{observation}

Since similar arguments are often needed for different graph classes, the following notation helps to avoid repetition:

\begin{definition}
    Let $\Psi$ be a graph class. A $\Psi$-cut is a vertex cut $M$ in the graph $G$ where the induced subgraph $G|_{M}$ belongs to the class $\Psi$. If the graph induced on the vertices of a cut is isomorphic to $H$, where $H$ is a particular graph, then we will call such a cut a $H$-cut.
\end{definition}

We introduce a general definition that allows us to prove several auxiliary statements:

\begin{definition}
\label{def:min_examples_family}
The family of minimal counterexamples $\mathcal{G}(\Psi, \alpha, \beta)$ is the set of all graphs $G$ satisfying:
\begin{itemize}
    \item $|G| \geq 4$,
    \item $q_{\alpha,\beta}(G) > 0$,
    \item $G$ has no $\Psi$-cut,
    \item $G$ has the smallest number of vertices among all graphs with the above three properties.
\end{itemize}
\end{definition}

Note that the first condition eliminates trivial cases. In the following, we will assume $\Psi=\mathfrak{F}$ or $\Psi=\mathfrak{B}$. Some proofs are identical for both graph classes. Since the counterexamples are different in general, it is not  always possible to prove a statement for one class and then use the fact that $\mathfrak{F} \subset \mathfrak{B}$.

\section{Properties of cuts and separations}
\label{sec:prop_cut}

In this section, we assume that $H$ is an arbitrary graph and $q(G)=q_{\alpha,\beta}(G)$ is a quality function satisfying conditions \eqref{eq:cond_ab}.

\begin{lemma}[``Modularity of the quality function''] \label{cost_modularity}
    If $(M, L, R)$ is a separation of the graph $H$, then \[q(H) = q(H|_{L \cup M}) + q(H|_{R \cup M}) - q(H|_{M}).\]
\end{lemma}

\begin{proof}
    By the inclusion-exclusion principle:        
    \[|H| = |L \cup M| + |R \cup M| - |M|.\] 
    Since $(M, L, R)$ is a separation of $H$, there are no edges between $L$ and $R$, hence: 
    \[E(H) = E(H|_{L \cup M}) \cup E(H|_{R \cup M}).\]
    Therefore,
    \[
    \begin{split}
    e(H) &= e(H|_{L \cup M}) + e(H|_{R \cup M}) - e(H|_{(L \cup M) \cap (R \cup M)})\\
    &= e(H|_{L \cup M}) + e(H|_{R \cup M}) - e(H|_{M}).
    \end{split}
    \]
    Consequently,
    \[
    \begin{split}
     q(H|_{L \cup M}) &+ q(H|_{R \cup M}) - q(H|_{M})  =\alpha(|L \cup M| + |R \cup M| - |M|)  \\
     &-(e(H|_{L \cup M}) +e(H|_{R \cup M}) - e(H|_{M})) - \beta = q(H). \qedhere
    \end{split} 
    \] 
\end{proof}

\begin{lemma} \label{cut_lemma_1}
    If $(M, L, R)$ is a separation of the graph $H$ and $S \subset (L \cup M)$ is a cut of $H|_{L \cup M}$, then either $S$ is a cut of $H$ or $S \cap M$ is a cut of $H|_{M}$.
\end{lemma}

\begin{figure}[ht]
    \centering
    \includegraphics[width=8cm]{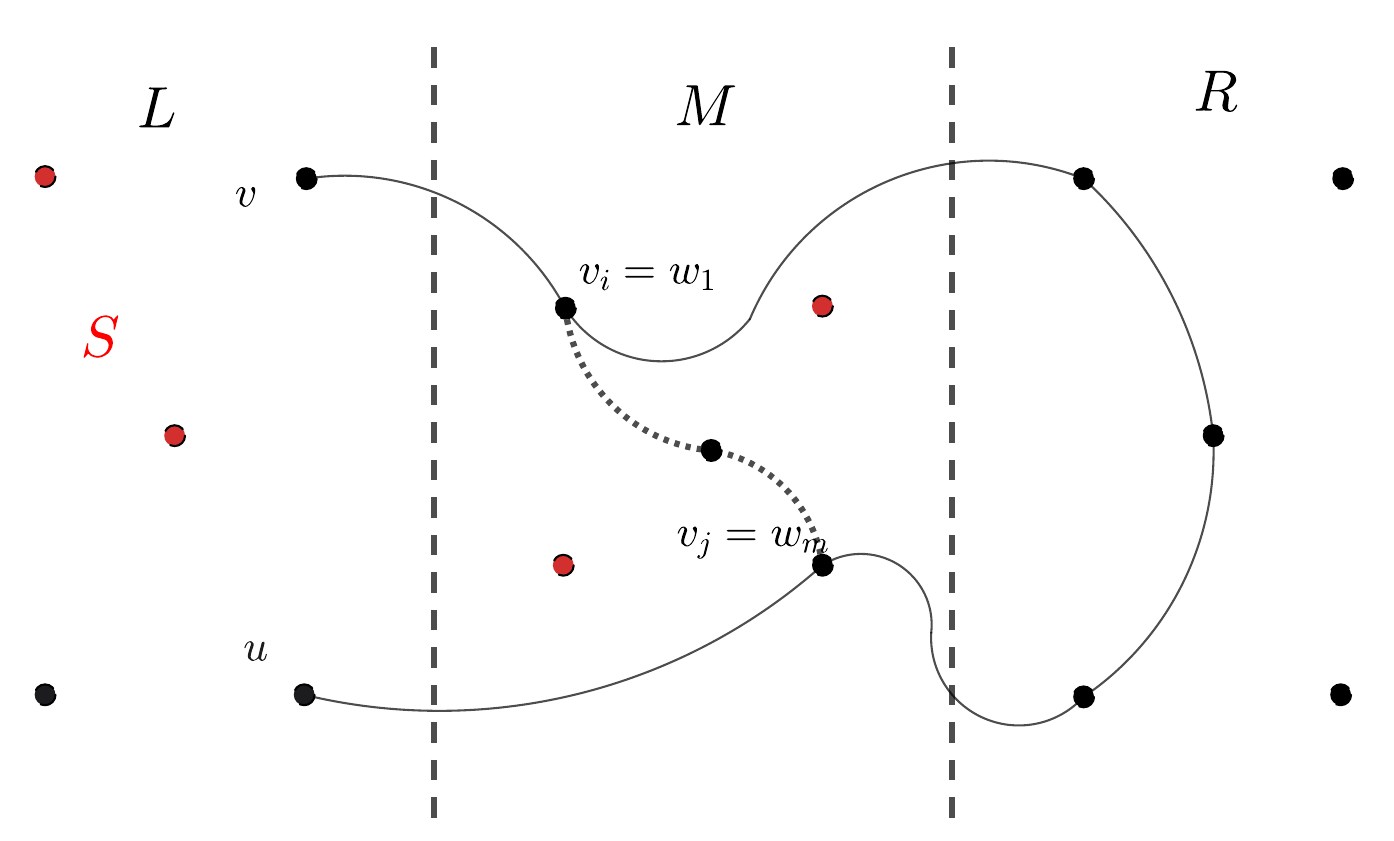}
    \label{fig:lemma2}
    \caption{Shortening the path in Lemma \ref{cut_lemma_1}. The cut vertices are colored red.}
\end{figure}

\begin{proof}
    Suppose that neither $S$ is a cut of $H$ nor $S \cap M$ is a cut of $H|_{M}$. Let $u$ and $v$ be vertices in different connected components of $H|_{(L \cup M) \setminus S}$. Let $v = v_1, v_2 \ldots, v_s=u$ be a path between them in $H|_{V(H) \setminus S}$ (which exists since $S$ is not a cut of $H$). By the choice of $u$ and $v$, this path passes through $R$ and, since $(M, L, R)$ is a separation of $H$, it must pass through $M$ (Fig.~\ref{fig:lemma2}). 
    
    Let $v_i$ and $v_j$ be the first and last vertices from $M$ on this path, respectively. Then all vertices before $v_i$ and after $v_j$ lie in $L \cup M$ (as $(M, L, R)$ is a separation). Let $v_i = w_1, w_2 \ldots, w_m = v_j$ 
    be a path between $v_i$ and $v_j$ in $H|_{M \setminus S}$ (which exists since $S \cap M$ is not a cut of $H|_{M}$). Then 
    \[v_1, \ldots, v_i=w_1, \ldots ,w_m=v_j, \ldots ,v_s\] 
    forms a path between $v$ and $u$ in $H|_{(L \cup M) \setminus S}$, contradicting the choice of $u$ and $v$.
\end{proof}

\begin{lemma} \label{cut_lemma_2}
    Let $(M, L, R)$ be a separation of graph $H$, $S_{L} \subset (L \cup M)$ is a cut of $H|_{L \cup M}$, and $S_{R} \subset (R \cup M)$ is a cut of $H|_{R \cup M}$. Furthermore, let $S_L \cap M = S_R \cap M$ and $|M \setminus S_L| = 2$. Then $S = S_L \cup S_R$ is a cut of $H$.
\end{lemma}

\begin{figure}[ht]
    \centering
    \includegraphics[width=8cm]{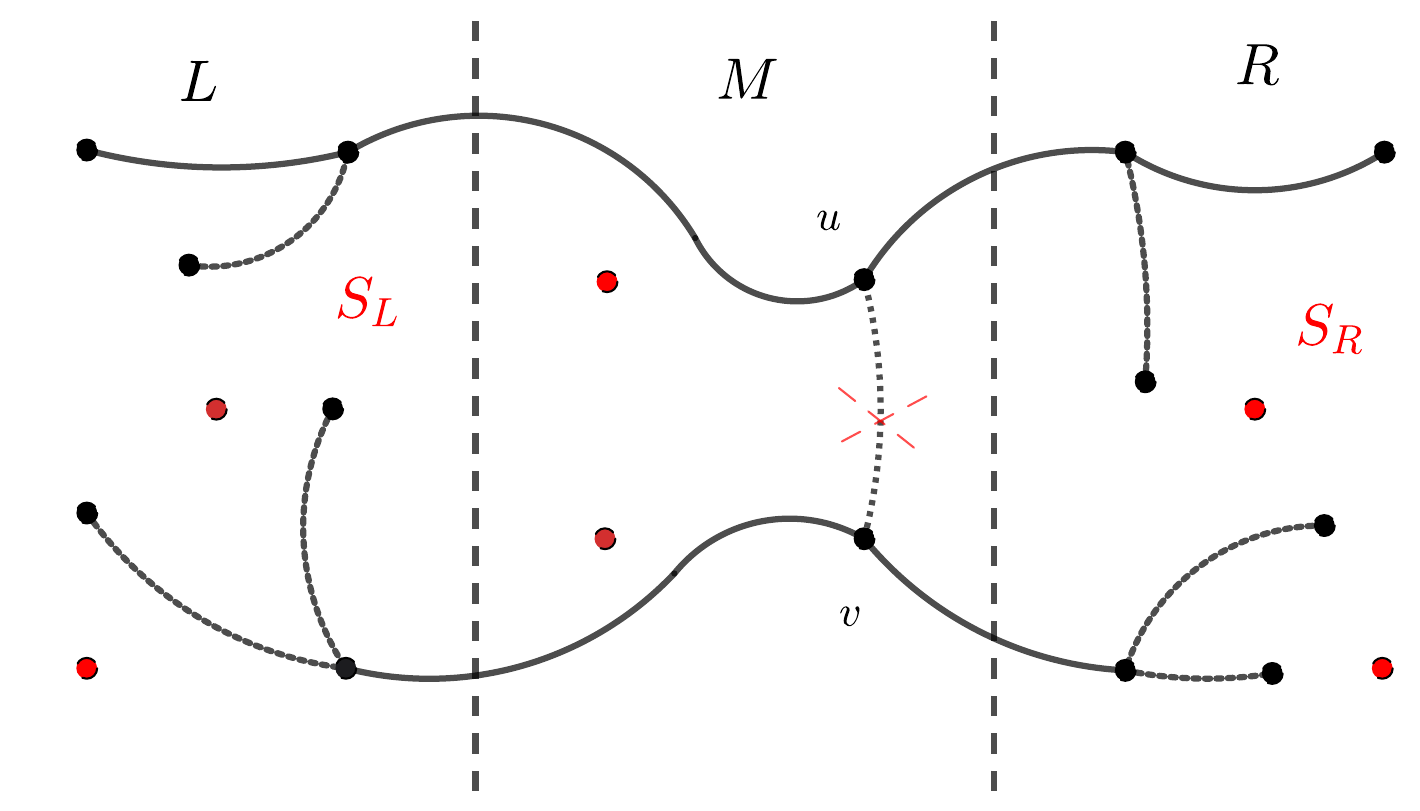}
    \label{fig:lemma3}
    \caption{Connected components in Lemma \ref{cut_lemma_2}.}
\end{figure}

\begin{proof}
     Let $u, v \in M \setminus S_L$. Observe  that $u$ and $v$ lie in different connected components of $H|_{(L \cup M) \setminus S_L}$. Indeed, if they were in the same component, then from any path in $H|_{V(H)\setminus S}$ between vertices in $(L \cup M) \setminus S_L$, we could remove the segment between the first and last intersection with $M$ (if necessary) and add a path from $u$ to $v$ in $H|_{(L \cup M) \setminus S_L}$, to get a path between them in $H|_{(L \cup M) \setminus S_L}$. This contradicts that $S_L$ is a cut of $H|_{L \cup M}$.   Similarly, $u$ and $v$ lie in different connected components of $H|_{(R \cup M) \setminus S_R}$ (Fig.~\ref{fig:lemma3}). 
     
     Let $L_u$ be the connected component of $u$ in $H|_{(L \cup M) \setminus S_L}$, and $R_u$ the connected component of $u$ in $H|_{(R \cup M) \setminus S_R}$. Since neither $L_u$ nor $R_u$ contains vertices from $M$ other than $u$, it follows that $L_u \cup R_u$ is the connected component of $u$ in $H|_{V(H)\setminus S}$, which does not contain $v$. Therefore, $H|_{V(H)\setminus S}$ is disconnected, which means that $S$ is indeed a cut of $H$.
\end{proof}


\section{Properties of the minimal counterexample}
\label{sec:min_counter}

In this section, we assume that the class of induced graphs on cuts is one of the two classes considered, $\Psi\in \{\mathfrak{B}, \mathfrak{F}\}$, and the quality function $q(\cdot)=q_{\alpha,\beta}(\cdot)$ satisfies conditions $\eqref{eq:cond_ab}$. 

Let the graph $G$ be a minimal counterexample, i.e.,
\begin{equation}
\label{eq:min_counter}
    G \in \mathcal{G}(\Psi, \alpha, \beta).
\end{equation}

Recall that Definition~\ref{def:min_examples_family} only considers counterexamples with at least 4 vertices. All statements in this section will be proved under the assumption~\eqref{eq:min_counter}. For brevity, we will not repeat this condition in each formulation.

\begin{lemma}
\label{prop:separation_cut}
    If $(M, L, R)$ is a separation of $G$, $|M|\geq 3$ and $q(M) \geq 0$, then either $G|_{L \cup M}$ or $G|_{R \cup M}$ has a $\Psi$-cut.
\end{lemma}

\begin{proof}
    Indeed, from Lemma~\ref{cost_modularity} and the properties of minimal counterexamples (see Definition~\ref{def:min_examples_family}), we have \[
    q(G|_{L\cup M})+q(G|_{R \cup M}) = q(G)+q(M)>0.
    \]
    At least one term on the left-hand side is positive. Since the number of vertices in $G|_{L\cup M}$ and $G|_{R \cup M}$ is less than in $G$ but not less than 4, at least one of them must contain a $\Psi$-cut.
\end{proof}

\begin{lemma} \label{small_cases}
    $|G|\geq 6$.
\end{lemma}

\begin{proof}
    Express the quality function and substitute $\beta = 4\alpha-6$ from Condition~\eqref{eq:cond_ab}:
    \[q(G) = \alpha|G| - e(G) - (4\alpha - 6) = \alpha(|G| - 4) - (e(G) - 6).\] 
    Furthermore, $\alpha \leq 3$. It is easy to verify that $q(K_4) = 0$, $q(K_5) < 0$, and $q(K_5\setminus e) \leq 0$, while all other graphs with at most 5 vertices have a forest cut (and thus also a bipartite cut).
\end{proof}

We are now ready to consider cuts of $G$ with at most four vertices.

\begin{lemma} \label{4_connectivity}
    $G$ is 4-connected and has no $K_4$-cut.
\end{lemma}

\begin{proof}
    Suppose not. Then there exists a separation $(M, L, R)$ of $G$ such that $|M| \leq 3$ or $G|_M = K_4$. Since $G$ has no $\Psi$-cuts, $G|_M$ is not a forest, hence $G|_M = K_3$ or $G|_M = K_4$. From the conditions on $\alpha$ and $\beta$, we get 
    \[q(K_4) = 4\alpha - \beta - 6 = 0,\] 
     \[q(K_3) = 3\alpha - \beta - 3 = 3\alpha - (4\alpha - 6) - 3 = 3 - \alpha \geq 0.\]
     In both cases $q(G|_M) \geq 0$, so by Lemma~\ref{prop:separation_cut}, either $G|_{L \cup M}$ or $G|_{R \cup M}$ has a $\Psi$-cut. Denote it by $S$. By Lemma~\ref{cut_lemma_1}, either $S \cap M$ is a cut of $G|_{M}$, or $S$ is a cut of $G$. The first case is impossible since $G|_{M}$ is complete, while the second contradicts the absence of $\Psi$-cuts in $G$.
\end{proof}

\begin{lemma} \label{no_t3}
    $G$ has no  $D$-cut.
\end{lemma}

\begin{proof}
    Suppose the contrary. Let $(M, L, R)$ be a separation of $G$ where $G|_{M}$ is isomorphic to $D$. Let $u, v$ be non-adjacent vertices in $M$. Since $q(G|_{M}) = 4\alpha - 5 - \beta = 1$, by Lemma~\ref{cost_modularity} we have 
    \[q(G|_{L\cup M}) + q(G|_{R\cup M}) = q(G) + 1 > 1.\]
    Therefore, either one of $q(G|_{L\cup M})$ and $q(G|_{R\cup M})$ is strictly greater than 1, or both are strictly greater than 0. We consider these two cases.
    
    \textbf{Case 1:} Suppose one of $q(G|_{L\cup M})$ and $q(G|_{R\cup M})$ (without loss of generality, $q(G|_{L\cup M})$) is strictly greater than 1. Let $G'$ be the graph obtained from $G$ by adding edge $uv$. Then $|L\cup M| \ge 4$ and 
    \[q(G'|_{L \cup M}) = q(G|_{L \cup M}) - 1 > 0.\]
    By the minimality of $G$, $G'|_{L \cup M}$ has a $\Psi$-cut, denoted $S$. Since $G'|_{M}$ is complete, $S \cap M$ is not a cut of $G'|_{M}$, so by Lemma~\ref{cut_lemma_2}, $S$ is a cut of $G'$. But then $S$ is a $\Psi$-cut of $G$, since removing an edge cannot make $S$ stop being a cut, and the induced subgraph still belongs to $\Psi$ by Observation~\ref{obs:hereditary}. This contradicts the absence of $\Psi$-cuts in $G$.
        
    \textbf{Case 2:} Suppose both $q(G|_{L\cup M})$ and $q(G|_{R\cup M})$ are positive. Then by the minimality of $G$  and since $|L\cup M| \geq 4, |R\cup M| \geq 4$, both $G|_{L\cup M}$ and $G|_{R\cup M}$ have $\Psi$-cuts. Denote them by $S_L$ and $S_R$ respectively. Since $G$ has no $\Psi$-cuts, $S_L$ is not a cut of $G$, so by Lemma~\ref{cut_lemma_1}, $S_L \cap M$ is a cut of $M$. Hence $S_L \cap M = M \setminus \{u, v\}$ (since this is the only cut in $D$). Similarly, $S_R \cap M = M \setminus \{u, v\}$. Then by Lemma~\ref{cut_lemma_2}, $S = S_L \cup S_R$ is a cut of $G$. It remains to note that $S$ is a $\Psi$-cut of $G$, since gluing two forests along an edge produces a forest, and gluing two bipartite graphs along an edge produces a bipartite graph (see Observation~\ref{obs:hereditary}).
\end{proof}

\begin{lemma} \label{no_c4}
    $G$ has no $C_4$-cut.
\end{lemma}

\begin{figure}[ht]
    \centering
    \includegraphics[width=8cm]{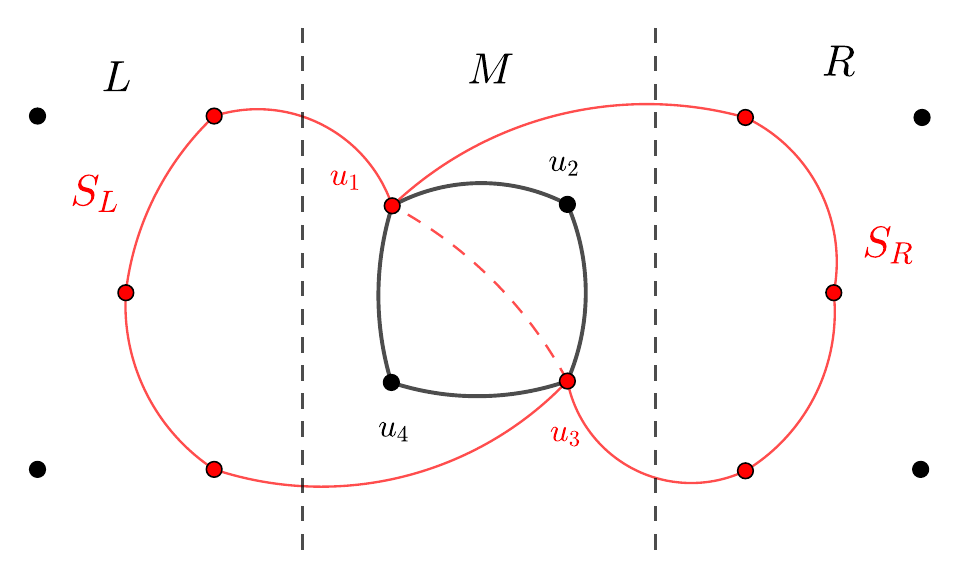}
    \label{fig:lemma8}
    \caption{Cycle on the vertices of the cut in Lemma \ref{no_c4}.}
\end{figure}

\begin{proof}
    The claim is trivial for the case of bipartite cuts since $C_4$ is bipartite. Consider the case of forest cuts ($\Psi=\mathfrak{F}$).
    
    Let $(M, L, R)$ be a separation of $G$ where $G|_{M}$ is isomorphic to $C_4$. Note that $q(G|_{M}) = 4\alpha - 4 - \beta = 2$. By Lemma~\ref{cost_modularity}, 
    \[q(G|_{L\cup M}) + q(G|_{R\cup M}) = q(G) + 2 > 2,\] 
    so either one of $q(G|_{L\cup M})$ and $q(G|_{R\cup M})$ is strictly greater than 2, or one is strictly greater than 1 and the other strictly greater than 0. The case where one exceeds 2 is handled similarly to Case 1 in Lemma~\ref{no_t3}.
    
    Consider the second case. Without loss of generality, assume that \[q(G|_{L\cup M}) > 1, \quad q(G|_{R\cup M}) > 0.\] By the minimality of $G$ and $|R\cup M| \geq 4$, $G|_{R\cup M}$ has a forest cut, denoted by $S_R$. From Lemma~\ref{cut_lemma_1} and the absence of forest cuts in $G$, $S_R \cap M$ must be a cut of $M$, i.e., $S_R \cap M$ consists of two non-adjacent vertices. 
    
    Let $u_1, u_3$ be these vertices, and $u_2, u_4$ the other two vertices in $M$. Consider the graph $G'$ obtained from $G$ by adding the edge $(u_1, u_3)$. Note that $S_R$ remains a cut of $G'|_{R \cup M}$, since $G$ and $G'$ differ only by an edge between vertices in $S_R$. Since 
    \[q(G'|_{L\cup M}) = q(G|_{L\cup M}) - 1 > 0,\] the subgraph $G'|_{L\cup M}$ has a forest cut, denoted $S_L$. If $S_L$ is a cut of $G'$, then it is also a cut of $G$, contradicting the absence of forest cuts in $G$. 
    
    Thus, by Lemma~\ref{cut_lemma_1}, $S_L \cap M$ is a cut of $G'|_M$, i.e., 
    \[S_L \cap M = \{u_1, u_3\} = S_R \cap M.\]
    By Lemma~\ref{cut_lemma_2}, $S = S_L \cup S_R$ is a cut of $G'$ and hence of $G$.
    
    To contradict the absence of forest cuts in $G$, it remains to check that $G|_S$ is a forest.
    
    Suppose the opposite. Then $G|_S$ contains a simple cycle (Fig.~\ref{fig:lemma8}). Since $G|_{S_L}$ and $G|_{S_R}$ are acyclic, the cycle must intersect both $S \cap L$ and $S \cap R$. However, $(\{u_1, u_3\}, S \cap L, S \cap R)$ is a separation of $S$, so the cycle contains a path from $u_1$ to $u_3$ which lies in $G|_{S_L}$. This path becomes a cycle when the edge $(u_1, u_3)$ is added, contradicting that $S_L$ is a forest cut of $G'|_{L \cup M}$.
\end{proof}

\begin{lemma} \label{4_cut_structure}
    If $S$ is a cut of $G$ and $|S| \leq 4$, then $G|_{S}$ is isomorphic to $T_0$ or $T_1$.
\end{lemma}

\begin{proof}
    By Lemma~\ref{4_connectivity}, $|S| = 4$. Since $G$ has no $\Psi$-cuts, $G|_S$ is not a forest, i.e., it is isomorphic to one of $C_4$, $T_0$, $T_1$, $D$, or $K_4$. But by Lemmas~\ref{4_connectivity}, \ref{no_t3}, and~\ref{no_c4}, it cannot be isomorphic to $K_4$, $D$, or $C_4$. Thus, it must be isomorphic to $T_0$ or $T_1$.
\end{proof}

Considering cuts formed by neighborhoods of vertices in $G$, from Lemmas~\ref{4_connectivity} and~\ref{4_cut_structure} we obtain:

\begin{corollary} \label{4_degree}
    $G$ has no vertices of degree less than 4, and the subgraph induced by the neighbors of any degree-4 vertex is isomorphic to $T_0$ or $T_1$.
\end{corollary}

The degree bound implies that a minimal counterexample cannot have 6 vertices.

\begin{lemma} \label{v_neq_6}
    $|G| \geq 7$.
\end{lemma}

\begin{proof}
    By Lemma~\ref{small_cases}, $|G| \geq 6$, so it remains to show that $|G| \neq 6$. Suppose $|G| = 6$. We derive a contradiction by bounding $e(G)|$ in two ways. First, since all vertex degrees are at least 4, $e(G) \geq 2|G|=12$. Second, $q(G) > 0$ implies 
    \[e(G) < \alpha |G| - \beta = 6\alpha - (4\alpha - 6) = 2\alpha + 6 \leq 12. \qedhere\]
\end{proof}

\begin{lemma} \label{neigh_not_full}
The neighborhood of any vertex in $G$ cannot be a complete graph.
\end{lemma}

\begin{proof}
Suppose $G$ has a vertex $u$ whose neighborhood $N(u)$ is a complete graph. Then $(N(u), V(G)\setminus(N(u) \cup \{u\}), \{u\})$ is a separation of $G$. Let $G' = G\setminus\{u\}$. Then $$q(G') = q(G) - \alpha + \operatorname{deg} u > q(G) > 0,$$ since Corollary~\ref{4_degree} gives $\operatorname{deg} u \geq 4 > \alpha$. Note also that $|G'| \geq 6$.

Since $G$ is the minimal graph of positive quality without a $\Psi$-cut, $G'$ must have a $\Psi$-cut $S$. By Lemma~\ref{cut_lemma_1}, either $S$ is a cut of $G$, or $S\cap N(u)$ is a cut of $N(u)$. The first is impossible since $G$ has no $\Psi$-cuts, and the second is impossible since $N(u)$ is complete. 
\end{proof}

\section{Improving the estimate for forest cuts}
\label{sec:forest}

In this section, we complete the proof of Theorem~\ref{thm:forest}.

We assume that the graph $G$ satisfies the following conditions:
\begin{equation}
\label{eq:cond_ab_forest}
2<\alpha \leq \frac{5}{2}, \; 4\alpha-\beta=6, \quad
G \in \mathcal{G}(\mathfrak{F},\alpha,\beta).
\end{equation}

Here we consider minimal counterexamples under slightly broader assumptions than needed for Theorem~\ref{thm:forest}. Since the constraints on $\alpha$ in \eqref{eq:cond_ab_forest} are stronger than in Condition~\eqref{eq:cond_ab}, we can use all statements from Section~\ref{sec:min_counter}.

The key part of the proof is the following lemma. For $\Psi=\mathfrak{B}$ the statement holds without additional constraints on $\alpha$, but the proof differs substantially, so the case of bipartite cuts case will be treated separately in the next section.

\begin{lemma} \label{tree_common_tetrahedron}
    Let~\eqref{eq:cond_ab_forest} hold. Then no two degree-4 vertices of $G$ lie in the same $K_4$ subgraph of $G$.
\end{lemma}

\begin{figure}[ht]
    \centering
    \includegraphics[width=6cm]{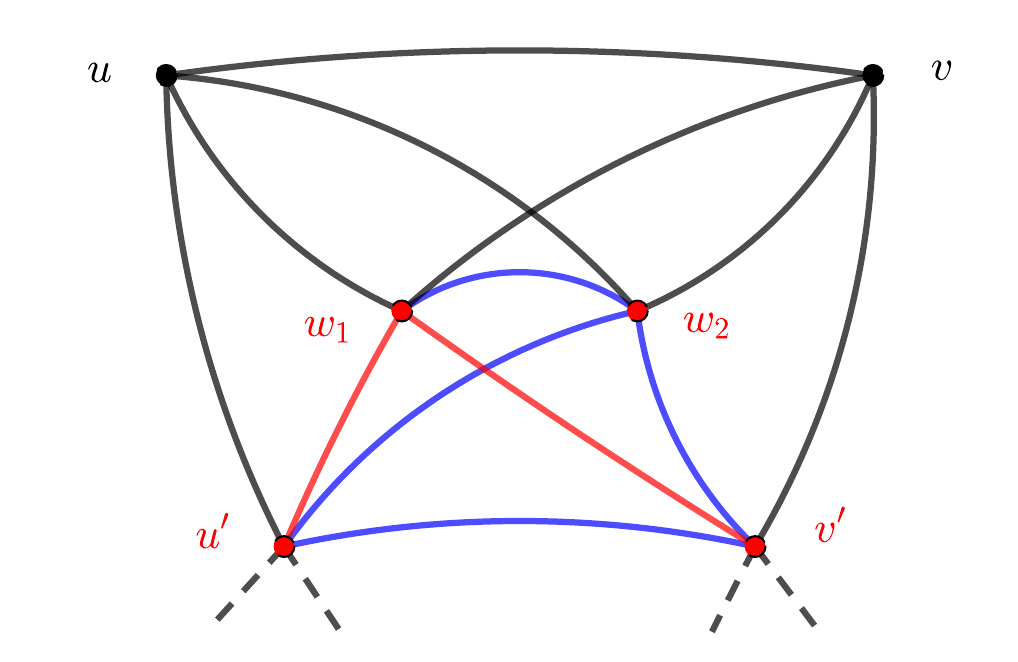}
    \label{fig:lemma12}
    \caption{Neighborhoods of $u$ and $v$ in Lemma \ref{tree_common_tetrahedron}.}
\end{figure}

\begin{proof}
    Suppose the contrary. Let $u, v$ be degree-4 vertices lying in $K_4 \subset G$. Let $w_1, w_2$ be the other two vertices of $K_4$, $u'$ the fourth neighbor of $u$, and $v'$ the fourth neighbor of $v$ (we allow $u'=v'$). Let 
    \[L = \{u, v\},\quad M = \{w_1, w_2, u', v'\}, \quad R = V(G) \setminus (L \cup M).\]
    Then $(M, L, R)$ is a separation of $G$. Note $R$ is non-empty since $|L \cup M| \leq 6$ while Lemma~\ref{v_neq_6} gives $|G| \geq 7$. Thus by Lemma~\ref{4_cut_structure}, $|M|=4$ (so $u'$ and $v'$ are distinct) and $G|_{M}$ is isomorphic to $T_0$ or $T_1$.
    
    The vertex $u'$ cannot be adjacent to both $w_1$ and $w_2$, since this would make $G|_{N(u)}$ have at least 5 edges, contradicting Corollary~\ref{4_degree}. Similarly, $v'$ cannot be adjacent to both $w_1$ and $w_2$, so edge $(w_1, w_2)$ cannot be part of a triangle in $G|_{M}$. Hence, $G|_{M}$ is isomorphic to $T_1$ instead of $T_0$, which means that $M$ has exactly two missing edges.
    
    Let $G'$ be obtained from $G$ by adding the two missing edges between vertices of $M$ (so $G'|_{M} = K_4$, Fig.~\ref{fig:lemma12}). Then $e(G'|_{M \cup R}) = e(G) - 5$, $|G'|_{M \cup R}| = |G| - 2$, and using $\alpha \leq \frac{5}{2}$ we get 
    \[q(G'|_{M \cup R}) = q(G) - 2\alpha + 5 \geq q(G) > 0.\]
    By minimality of $G$, the graph $G'|_{M \cup R}$ has a forest cut. From Lemma~\ref{cut_lemma_1} and $G'|_{M}$ being complete, $S$ is a cut of $G'$ and thus of $G$, contradicting the absence of forest cuts in $G$.
\end{proof}

\begin{corollary}
\label{cor:forest_degree4}
Under the same assumptions, if $U$ is the set of degree-4 vertices in $G$, then $G|_{U}$ consists of isolated vertices and edges.
\end{corollary}

\begin{proof}
By Corollary~\ref{4_degree}, for any $u \in U$ the graph $G|_{N(u)}$ is isomorphic to either $T_0$ or $T_1$, i.e., contains a triangle plus a vertex of degree $0$ or $1$. If $u,v \in U$ are adjacent, $v$ cannot be part of the triangle in $G|_{N(u)}$ as this would create $K_4$ containing two degree-4 vertices, contradicting Lemma~\ref{tree_common_tetrahedron}. Hence for each adjacent pair $u,v \in U$, the vertex $v$ must have degree $0$ or $1$ in $G|_{N(u)}$ and vice versa. Therefore, edges in $G|_U$ form a matching.
\end{proof}

\begin{remark}
The following lemma is more general than Corollary~\ref{cor:forest_degree4} and will also be used later in the proof concerning bipartite cuts. In the bipartite case, the structural properties of $G|_U$ analogous to Corollary~\ref{cor:forest_degree4} will be proved in a different way.
\end{remark}

\begin{lemma} \label{independent_to_forest}
    Let \eqref{eq:min_counter} hold. Suppose that $U$ is the set of all degree-4 vertices in $G$, and $G|_{U}$ is a disjoint union of isolated vertices and edges. If the subgraph $G|_{V(G) \setminus U}$ has an independent cut, then $G$ admits a forest cut.
\end{lemma}

\begin{figure}[ht]
    \centering
    \includegraphics[width=8cm]{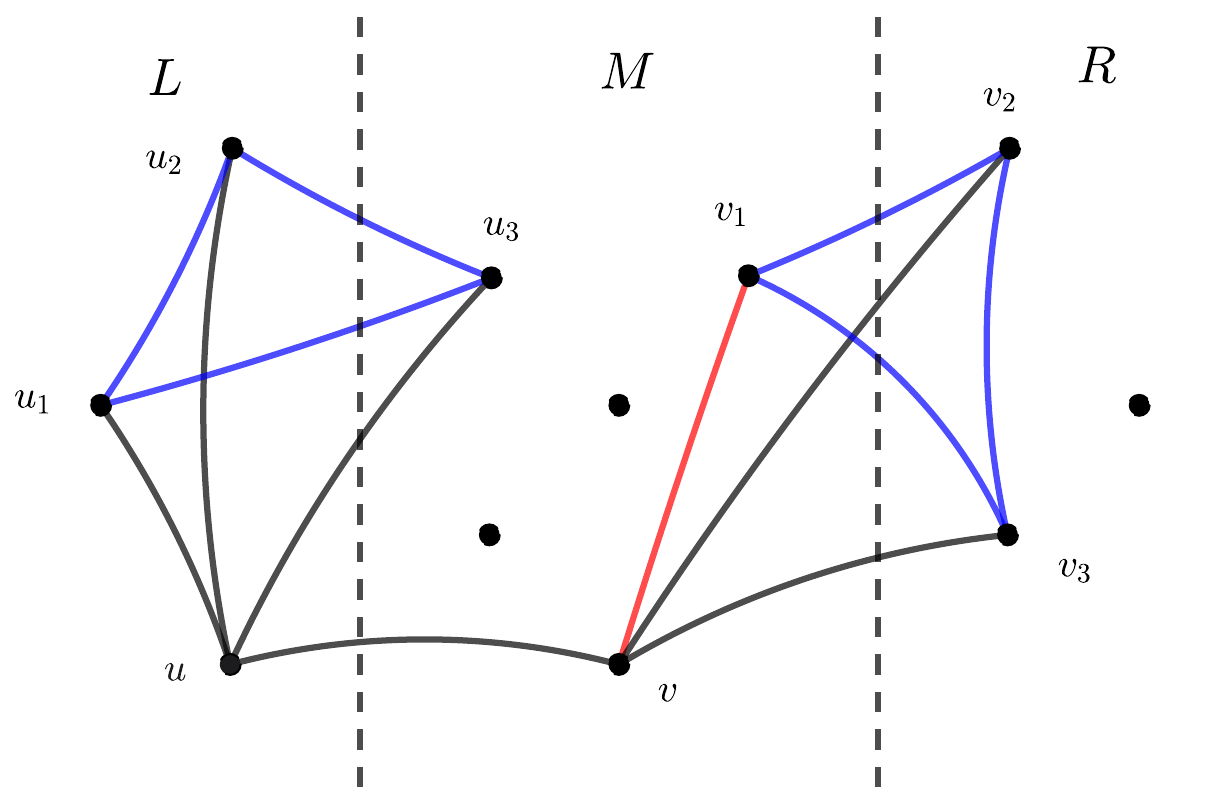}
    \caption{Adding two adjacent degree-4 vertices in Lemma~\ref{independent_to_forest}.}
    \label{fig:lemma13}
\end{figure}




\begin{proof}
    Let $M$ be an independent cut of $G|_{V(G) \setminus U}$ and $(M, L, R)$ be a corresponding separation. We process each component of $G|_U$ (isolated vertices or edges) one by one, adding the considered vertices to $L$, $R$ or $M$ while maintaining that $(M', L', R')$ is a separation, and $G|_{M'}$ is a forest (though not necessarily an independent set) in $G|_{M' \cup L' \cup R'}$. 
    
    Note that 
    at each step, any edge appearing in $G|_M'$ is incident to a vertex of degree 4, so when adding a connectivity component of $G_U$, the added vertex cannot have two adjacent neighbors in $M'$.
    
    We consider vertex and edge additions separately.
    
    \textbf{Case 1:} Adding an isolated (in $G|_U)$ vertex $u$. Let $w_1, w_2, w_3$ be its neighbors forming a triangle, and $v$ its fourth neighbor. At most one of $w_1, w_2, w_3$ lies in $M$, while the others lie in the same set $L$ or $R$ (as there are no edges between $L$ and $R$). Without loss of generality, assume that they lie in $L$. If $v \notin R$, add $u$ to $L$. Let $L'=L+\{u\}$. Since no neighbors of $u$ lie in $R$, the triple $(M, L', R)$ remains a separation and $G|_M$ remains a forest, since $M$ is unchanged. If $v \in R$, add $u$ to $M$, let $M'=M+\{u\}$. Then $(M', L, R)$ remains a separation (as $L$ and $R$ are unchanged) and $G|_M$ remains a forest since $u$ has at most one neighbor in $M$.
    
    \textbf{Case 2:} Adding two adjacent degree-4 vertices $u, v$. Let $u_1, u_2, u_3$ be neighbors of $u$ other than $v$, and $v_1, v_2, v_3$ be neighbors of $v$ other than $u$ (some $u_i$ and $v_i$ may coincide). As before, at most one of $u_1, u_2, u_3$ lies in $M$, with the others in $L$ (without loss of generality). Similarly, at most one of $v_1, v_2, v_3$ lies in $M$. Add $u$ to $L$ and $v$ to $M$. Let $L'=L+\{u\}$, $M'=M+\{v\}$. Vertex $u$ has no neighbors in $R$, so $(M', L', R)$ remains a separation, and $v$ has at most one neighbor in $M$, so $G|_{M'}$ remains a forest (Fig.~\ref{fig:lemma13}).
\end{proof}

We now prove Theorem \ref{thm:forest}.

\begin{proof}
    Suppose the contrary. Let $\alpha = \frac{19}{8} < \frac{5}{2}, \beta = \frac{28}{8}$ and $G \in \mathcal{G}(\mathfrak{F},\alpha,\beta)$ be a minimal counterexample. Let $U$ be the set of degree-4 vertices in $G$, and $G' = G|_{V(G) \setminus U}$. By Lemma~\ref{independent_to_forest}, the absence of forest cuts in $G$ implies that $G'$ has no independent cut, so by Theorem~\ref{thm:independent}:
    \[e(G') \geq 2|G'| - 3.\]
    
    Since all vertices in $U$ have degree 4 and $G|_U$ has at most $\frac{|U|}{2}$ edges:
    \[e(G') = e(G) - 4|U| + e(G|_M) \leq e(G) - \frac{7}{2}|U|.\]
    
    Combine these inequalities:
    \begin{equation}    
    \label{eq:thm2a}
        e(G) \geq 2|G'| - 3 + \frac{7}{2}|U| = 2|G| + \frac{3}{2}|U| - 3.
    \end{equation}
    On the other hand, since all vertices in $V(G) \setminus U$ have degree at least 5:
    \begin{equation}
    \label{eq:thm2b}
        e(G) \geq \frac{4|U| + 5(|G| - |U|)}{2} = \frac{5}{2}|G| - \frac{|U|}{2}.
    \end{equation}
    Taking a weighted sum, namely \eqref{eq:thm2a} with coefficient~$\frac{1}{4}$ and \eqref{eq:thm2b} with~$\frac{3}{4}$, gives:
    $$e(G) \geq \frac{19}{8}|G| - \frac{3}{4},$$
    which contradicts the condition $e(G) < \alpha |G| - \beta$.
\end{proof}

\begin{remark}
    \label{rmk:margin}
    One might think that the estimate could be improved, since the contradiction is reached with some margin. However, the value of $\beta$ in Theorem~\ref{thm:forest} is determined by the condition $4\alpha-\beta=6$, which was used many times before. A similar situation will arise in the proof of Theorem~\ref{thm:bipartite}.
\end{remark}

\section{Estimate for bipartite cuts}
\label{sec:bipartite}

In this section, we prove Theorem~\ref{thm:bipartite}. We assume that conditions \eqref{eq:cond_ab} hold, and that $G = (V,E)$ is a minimal counterexample for bipartite cuts, i.e.,
\[
    G \in \mathcal{G}(\mathfrak{B},\alpha,\beta).
\]

Additional constraints on the parameter $\alpha$ will not be needed for auxiliary statements, except in the last cases where parameter values are taken directly from Theorem~\ref{thm:bipartite}.

As in the previous section, we study properties of the subgraph $G|_U$ induced on the set $U$ of degree-4 vertices. This requires a lemma specific to the bipartite case:

\begin{lemma}\label{k23_contraction}
The graph $G$ does not contain a pair of non-adjacent vertices with at least three common neighbors.
\end{lemma}

\begin{proof}
Suppose there exist such vertices $u_1$ and $u_2$ with common neighbors $v_1, v_2, v_3$. Construct $G'$ from $G$ by: removing $u_1, u_2$ and their edges, then adding vertex $u$ connected to all neighbors of $u_1$ or $u_2$. Note that
\[q(G') \geq q(G) - \alpha + 3 \geq q(G) > 0.\]
Since $G'$ has fewer vertices than $G$ but more than four, $G'$ must have a bipartite cut $S'$. It must contain $u$, otherwise it would be a cut in $G$. Then $S = S'\cup\{u_1, u_2\} \setminus u$ is a bipartite cut in $G$ (placing $u_1, u_2$ in the partition of $u$), contradicting the properties of $G$.
\end{proof}

We need to prove that the subgraph induced on degree-4 vertices consists of a matching and isolated vertices. In the next two statements we show that no three degree-4 vertices form a triangle.

\begin{lemma} \label{4_no_tethraedron}
No two degree-4 vertices of G lie in the same
$K_4$ subgraph~of~$G$.
\end{lemma}

\begin{proof}
Suppose such vertices $u, v$ exist. Let $x, y$ be their common neighbors, $u_1$ and $v_1$ be the fourth neighbors of $u$ and $v$ respectively. The vertices $u_1$ and $v_1$ are distinct, otherwise $G$ has a 3-vertex cut $\{x, y, u_1\}$ which contradicts Lemma~\ref{4_cut_structure}.  Neither of $u_1$ nor $v_1$ can be adjacent to both $x, y$ since $N(u)$ and $N(v)$ are isomorphic to either $T_0$ or $T_1$. Since $M = \{x, y, u_1, v_1\}$ is a cut, $G|_{M}$ contains a triangle. Without loss of generality, assume that $u_1, v_1, y$ form a triangle. Then $v_1$ and $u$ are non-adjacent and have common neighbors $v, u_1, y$ which contradicts Lemma \ref{k23_contraction}.
\end{proof}

\begin{lemma} \label{4_no_triangle}
$G$ does not contain a triangle of three degree-4 vertices.
\end{lemma}

\begin{proof}
Suppose such vertices $u, v, w$ exist. Then $G|_{N(u)}$ contains a triangle. Without loss of generality, $v$ lies in this triangle and together with $u$ this triangle forms a copy of $K_4$, which contradicts Lemma~\ref{4_no_tethraedron}.
\end{proof}





\begin{lemma}
The maximum degree in $G|_U$ is at most 1.
\end{lemma}

\begin{proof}
Suppose degree-4 vertices $u, v, w$ exist with $v$ adjacent to both $u,w$. By Lemma~\ref{4_no_triangle}, $u$ and $w$ are not adjacent. Let $x_1,x_2$ be the other neighbors of $v$. Some three neighbors of $v$ must form a triangle  (say $w, x_1, x_2$). Let $y$ be the fourth neighbor of $w$. Then $M=\{u,x_1,x_2,y\}$ is a cut separating $v,w$ from the rest of the graph. $G|_M$ must be bipartite: $u$ has at most one neighbor in $\{x_1,x_2,y\}$ (otherwise $u,w$ would have three common neighbors, contradicting Lemma~\ref{k23_contraction}). Any triangle in $G|_M$ would require edges between $x_1,x_2,y$, but this would create two triangles in $N(w)$, contradicting Corollary~\ref{4_degree}. So we have found a bipartite cut, a contradiction.
\end{proof}

Thus, each degree-4 vertex has at most one degree-4 neighbor. We have shown that $G|_U$ is a matching plus isolated vertices. 

For $X \subset V = V(G)$, let $f(X)=e(X,V \setminus X)$ be the number of edges between $X$ and its complement. Let $A$ be a maximal independent set maximizing $m:=f(A)$, and $B=V\setminus A$. For convenience, let us list the variables that will be used in further reasoning.

\begin{notation}
Let:
\begin{itemize}
    \item $n=|G|$,
    \item $m = e(A, V \setminus A) =  \max e(X,V\setminus X)$ over independent sets $X \subset V$,
    \item $k = |A|$,
    \item $x$ is a number of degree-4 vertices in $A$,
    \item $y$ is a number of degree-5 vertices in $A$,
    \item $z$ is a number of degree-4 vertices in $A$ with a degree-4 neighbor.
\end{itemize}
\end{notation}

Clearly, $z \leq x$. We bound $e(G)$ in terms of $m,n,k$:

\begin{lemma}
\label{lm:edge_estimate_1}
$e(G)\geq m + 2n-2k-3.$
\end{lemma}

\begin{proof}
$G|_B$ has no independent cut (it would combine with $A$ to form a bipartite cut of $G$). By Theorem~\ref{thm:independent}:
\begin{equation}
    \label{eq:B_edges}
    e(G|_B) \geq 2(n-k)-3,
\end{equation}
and each vertex from $A$ has a neighbor from $B$, from which follows the statement.
\end{proof}

Next, we estimate $e(G)$ using the structure of degree-4 vertices:

\begin{lemma} \label{bip_edges_1}
$e(G)\geq 2n+3z - 3$.
\end{lemma}

\begin{proof}
First, let us note that the proof of Lemma~\ref{independent_to_forest} also works in the bipartite case since we only use the structure of a graph induced on degree-4 vertices. 

Now remove the $z$ degree-4 vertices in A and their degree-4 neighbors ($2z$ vertices, $7z$ edges). If the remaining graph has less than $2(n - 2z) - 3$ edges, then, following the argument of Lemma~\ref{independent_to_forest}, $G$ has a forest cut. Thus 
\[
    e(G) \geqslant 7z + 2(n - 2z) - 3 = 2n + 3z - 3. \qedhere
\]
\end{proof}

Finally, we estimate $m = e(A, B)$ in two ways.

\begin{lemma}
$m\geq 6k-2x-y$.
\end{lemma}

\begin{proof}
The independent set $A$ consists of $x$ vertices of degree 4, $y$ vertices of degree 5, and $k - x - y$ vertices of degree 6 or higher. Therefore:
\[
    m = e(A, B) \geq 4x + 5y + 6(k - x - y) = 6k - 2x - y. \qedhere 
\]
\end{proof}

Using \eqref{eq:B_edges}, we have:

\begin{corollary} \label{bip_edges_2}
$e(G)\geq 2n+4k-2x-y-3$.
\end{corollary}

\begin{observation}
\label{obs:red_1}
Every vertex in $B$ is connected to at least one vertex in $A$, otherwise the independent set $A$ could be expanded. Furthermore, if $u \in A$, $v\in B$ are connected by an edge and $\deg(u) < \deg(v)$, then $v$ has at least one other neighbor in $A$.
\end{observation}
Otherwise, we could replace $u$ with $v$ in $A$, maintaining independence while increasing $e(A, B)$.

\begin{definition}
    Let us color the edge connecting $A$ and $B$ red if its endpoint in $B$ has at least two neighbors in $A$.
\end{definition}

\begin{observation} 
\label{obs:red_2}
A degree-4 vertex in $A$ has:
\begin{itemize}
    \item At least 3 red edges if it has a degree-4 neighbor.
    \item Exactly 4 red edges otherwise.
\end{itemize}
\end{observation}

\begin{lemma}
\label{lm:deg5_red}
For $\alpha=\frac{80}{31}$, $\beta=\frac{134}{31}$, where $G$ is a minimal counterexample and the vertex sets $A$, $B$ are defined under these assumptions, every degree-5 vertex in $A$ has at least 2 red edges.
\end{lemma}

\begin{proof}
First, note that every $u \in A$ has at least one red edge. By Lemma~\ref{neigh_not_full}, $u$ has non-adjacent neighbors $v, w \in B$. At least one of $uv$ or $uw$ must be red. Otherwise we could remove $u$ from $A$ and add both $v$ and $w$, maintaining independence while increasing $e(A,B)$.

Now suppose some $u \in A$ has exactly one red edge $uv$. Let \[N(u) = \{v, u_1, u_2, u_3, u_4\}.\] The subgraph induced on $\{u_1, u_2, u_3, u_4\}$ must be complete. Otherwise we could find another red edge from $u$ by similar reasoning. Moreover, each $u_i$ has degree exactly 5.

Consider the subgraph $G|_L$ where $L = \{u, u_1, u_2, u_3, u_4\}$ (a complete graph). Let $v, v_1, v_2, v_3, v_4$ be the neighbors of these vertices outside $L$ (possibly coinciding). We analyze cases based on the number of distinct vertices in $M = \{v, v_1, v_2, v_3, v_4\}$:

\textbf{Case 1:} $|M| \leq 2$. Then $V(G) = M \cup L$ (otherwise $M$ would be a cut, which is impossible since all cuts have cardinality $\geq 4$). But then:
$$q(G) \leq \frac{80}{31}\cdot 7 - 15 - \frac{134}{31} < 0.$$
contradicting the definition of $G$.

\textbf{Case 2:} $|M| = 3$. Similarly $V(G) = M \cup L$. Each $v \in M$ has $\leq 3$ neighbors in $L$ and degree $\geq 4$, so $G|_M$ has at least 2 edges. Thus $e(G) \geq 17$ and:
$$q(G) \leq \frac{80}{31}\cdot 8 - 17 - \frac{134}{31} < 0,$$
a contradiction with the definition of $G$.

\textbf{Case 3:} $|M| = 4$. Without loss of generality, $v = v_1$. Then $M' = \{v, v_2, u_3, u_4\}$ is a bipartite cut (only two possible edges: $vv_2$ and $u_3u_4$) separating $u,u_1,u_2$, a contradiction with the definition of $G$.

\textbf{Case 4:} $|M| = 5$ with $G|_M$ complete. For $G' = G|_{V(G)\setminus L}$:
$$q(G') = q(G) - 5\alpha + 15 > q(G),$$
so $G'$ must have a bipartite cut $M'$, which would also be a cut in $G$. Indeed, suppose there are vertices $x$ and $y$ which lie in different connected components of $G'|_{V(G')\setminus M'}$, but between which there is a path $x=x_0, x_1, \ldots, x_n = y$ in the graph $G|_{V(G)\setminus M'}$. Along the way, we must encounter vertices from $L$. Consider $x_i$ and $x_j$, the first and last vertex from $L$. The vertices $x_{i-1}$ and $x_{j+1}$ are in the set $M$. Therefore, they are connected by an edge, and hence there is a path between $x$ and $y$ that avoids vertices from $L$, which contradicts the choice of $M'$.

\textbf{Case 5:} $|M| = 5$, and $G|_M$  is not complete. Without loss of generality, edge $vv_1$ is missing. Then $\{v, v_1, v_2, u_3, u_4\}$ is a bipartite cut (possible edges: $vv_2, v_1v_2, u_3u_4$), again contradicting the definition of $G$.

Thus, all cases lead to contradictions, proving the lemma.
\end{proof}

Using Observation~\ref{obs:red_2}, we obtain:

\begin{corollary}
Let $r$ be the number of red edges. Then:
$$r \geq 3z + 4(x-z) + 2y + k - x - y.$$
\end{corollary}

\begin{lemma}
$2m\geq 2n-k+3x+y-z$.
\end{lemma}

\begin{proof}
We estimate $m$ using the number of red edges. Let $r_v$ be the number of red edges from vertex $v$. Each $v \in B$ has either:
\begin{itemize}
    \item exactly one edge to $A$, or
    \item at least $r_v \geq 2$ red edges.
\end{itemize}
Thus:
\[
\begin{split}
m &\geq \sum_{v\in B}\left(1 + \frac{r_v}{2}\right) \geq n - k + \frac{r}{2}  \\
&= n - k + \frac{1}{2}(3z + 4(x-z) + 2y + k - x - y).
\end{split}
\]

Multiplying by 2 gives:
\[
2m \geq 2n - k + 3x + y - z. \qedhere
\]
\end{proof}

Taking into account \eqref{eq:B_edges}, we obtain:

\begin{corollary} \label{bip_edges_3}
$2e(G)\geq 6n-5k+3x+y-z-6$.
\end{corollary}

We can now complete the proof of Theorem~\ref{thm:bipartite}.

\begin{proof}
Lemma~\ref{bip_edges_1} established the estimate \begin{equation}
    \label{eq:final1}
    e(G) \geq 2n + 3z - 3.
\end{equation}
Corollaries~\ref{bip_edges_2} and~\ref{bip_edges_3} provided the estimates 
\begin{equation}
    \label{eq:final2}
    e(G) \geq 2n + 4k -2x - y - 3
\end{equation}
and 
\begin{equation}
    \label{eq:final3}
    2e(G) \geq 6n - 5k + 3x + y - z - 6.
\end{equation}

Taking a weighted sum of \eqref{eq:final1}, \eqref{eq:final2} and \eqref{eq:final3} with coefficients 1, 12, and 9 respectively, we have:

\[
\begin{split}
31e(G) &= e(G) + 12e(G) + 9\cdot2e(G) \\
&\geq 2n + 3z - 3 + 12(2n+4k-2x-y-3) \\
&\quad + 9(6n-5k+3x+y-z-6) \\
&= 80n + 3k - 6z + 3x - 3y - 93 \\
&> 80n - 134,
\end{split}
\]
where the last inequality holds because $3k + 3x - 6z - 3y \geq 0$ follows from $k \geq z+y$ and $x \geq z$.

Thus, we conclude that $e(G) > \frac{80}{31}n - \frac{134}{31}$, which contradicts our initial assumption about $G$.
\end{proof}

\section{Conclusion}

The results presented in this paper required an extensive search of systems of inequalities and corresponding linear programming problems. This work can be continued, and the estimates can be further improved, but we believe it unlikely that such an approach will lead to a proof of Conjecture \ref{cnj:main} unless fundamentally different ideas are used. 






\section*{Acknowlegments}

This work was basically done during the workshop ``Open Problems in Combinatorics and Geometry IV'', Adygea, Dakhovskaya, October 13-27, 2024. The workshop was organised by the Moscow Institute of Physics and Technology with the support of the Caucasus Mathematical Center of Adyghe State University.

The authors are grateful to Alexander Polyanskii for highlighting this problem.

\bibliographystyle{abbrv}
\bibliography{main}

	

\end{document}